\documentclass[a4paper,10pt]{article}

\usepackage{amssymb}
\usepackage{graphicx}
\usepackage{bbm}
\usepackage{mathrsfs}
\usepackage{amsmath}
\usepackage{amsthm}
\usepackage[center]{caption}
\usepackage{enumerate}
\usepackage{verbatim}
\usepackage{authblk}
\usepackage[hmargin=3cm, vmargin=3.5cm]{geometry}
\usepackage{constants}
\usepackage[UKenglish]{babel}

\theoremstyle{plain}
\newtheorem{thm}{Theorem}[]
\newtheorem{cor}[thm]{Corollary}
\newtheorem{lem}[thm]{Lemma}
\newtheorem{prop}[thm]{Proposition}

\theoremstyle{definition}

\renewcommand{\P}{\mathbb{P}}
\newcommand{\E}{\mathbb{E}}

\newcommand{\ind}{\mathbbm{1}}

\newcommand{\eps}{\varepsilon}
\newcommand{\bp}{\begin{proof}}
\newcommand{\ep}{\end{proof}}
\newcommand{\Pnl}{\mathbb{P}_{n,\lambda}}
\newcommand{\Enl}{\mathbb{E}_{n,\lambda}}
\newcommand{\pnl}{p_{n,\lambda}}
\newcommand{\Cc}{\mathcal{C}}
\def\bal#1\eal{\begin{align*}#1\end{align*}}

\newconstantfamily{error}{symbol=\mathcal{E}}

\author{Matthew I. Roberts\thanks{University of Bath, Department of Mathematical Sciences, Bath BA2 7AY, UK. \texttt{mattiroberts@gmail.com}}}

\title{The probability of unusually large components in\\ the near-critical Erd\H{o}s-R\'enyi graph}

\begin{document}
\maketitle

\begin{abstract}
The largest components of the critical Erd\H{o}s-R\'enyi graph, $G(n,p)$ with $p=1/n$, have size of order $n^{2/3}$ with high probability. We give detailed asymptotics for the probability that there is an unusually large component, i.e.~of size $an^{2/3}$ for large $a$. Our results, which extend work of Pittel, allow $a$ to depend upon $n$ and also hold for a range of values of $p$ around $1/n$. We also provide asymptotics for the distribution of the size of the component containing a particular vertex.

\vspace{2mm}

\emph{MSC class:} 60C05, 05C80.
\end{abstract}

\section{Introduction}\label{sec:intro}

The Erd\H{o}s-R\'enyi graph $G(n,p)$ is perhaps the simplest interesting random graph model. We take $n$ vertices and connect each pair by an edge with probability $p$, independently of all other pairs. If $p=c/n$, then the graph undergoes a phase transition as $c$ passes $1$; for $c<1$, the largest connected components are of order $\log n$, whereas for $c>1$ the largest component is instead of order $n$ (both with probability tending to $1$ as $n\to\infty$). See for example \cite{bollobas_book} for more details.

When $p$ is close to $1/n$ the graph shows very different behaviour. Within the \emph{critical window}, i.e.~when $p=\frac{1}{n}(1+\lambda n^{-1/3})$ for constant $\lambda$, the size of the biggest component is of order $n^{2/3}$ with probability tending to $1$ as $n\to\infty$. Much more is known, and some of the most precise estimates are due to Pittel \cite{pittel:largest_cpt_rg}. In this article we are interested in the asymptotic probability that the largest component is unusually large, i.e.~$\P(L_1\ge an^{2/3})$ where $L_1$ is the size (number of vertices) of the largest component and $a\to\infty$. In \cite[Proposition 2]{pittel:largest_cpt_rg}, Pittel gave an asymptotic for $\lim_{n\to\infty} \P(L_1\ge an^{2/3})$ as $a\to\infty$ for fixed $\lambda$.

In fact the method that Pittel used to obtain his result is substantially more robust than the result itself implies; the asymptotic holds for a wider range of values of $\lambda$ and $a$, which may for example go to infinity with $n$. A similar method can also be used to approximate $\P(L_1=k)$ for $k\approx an^{2/3}$ with $a$ large. The aim of this article is to make these more general results precise, and rehash the methods of Pittel (which in turn owe much to another paper of \L{}uczak, Pittel and Wierman \cite{luczak_et_al:structure_RG}) to prove them. Our proofs consist largely of elementary approximations applied very precisely, but we believe the results could be useful given the increasing interest in near-critical random graph models and therefore are worth writing clearly in a general form. Indeed we have already used them in work on a dynamical version of the Erd\H{o}s-R\'enyi graph \cite{roberts_sengul:dynamical_ER}.

We should remark that Pittel proved several other results in his article \cite{pittel:largest_cpt_rg}, which we make no attempt to rework. Our results are based around his Proposition 2. We also note here that the constant $(2\pi)^{1/2}$ in the denominator of Pittel's result should instead be $(9\pi/8)^{1/2}$; there is a small oversight between the last line of page 266 and the first line of page 267 in \cite{pittel:largest_cpt_rg}.

\section{Results and further discussion}

\subsection{Main results}

Let $\pnl = n^{-1} + \lambda n^{-4/3}$, and let $\Pnl$ be the law of an Erd\H{o}s-R\'enyi random graph with $n$ vertices and edge probability $\pnl$. (We allow $\lambda$ to depend on $n$.) To state our main results, define
\[G_\lambda(x) = \frac{x^3}{8} - \frac{\lambda x^2}{2} + \frac{\lambda^2 x}{2} = \frac{x}{8}(x-2\lambda)^2,\]
\[\Cl[error]{M1err} = \Cr{M1err}(k,n,\lambda) = \frac{k^4}{n^3} + \frac{|\lambda|}{n^{1/12}} + \frac{|\lambda|^3 k}{n} + \frac{n^{2/3}}{k} + \frac{1}{n^{1/10}}\]
and
\[\Cl[error]{A1err} = \Cr{A1err}(n) = n^{1/4}e^{-n^{1/4}/80}.\]

\begin{thm}\label{thm:YZ}
Suppose that $-n^{1/12}\le \lambda \le \frac15 n^{1/12}$ and $(3\lambda\wedge 1)n^{2/3}\le k \le n^{3/4}$. Let $a=k/n^{2/3}$. Then
\begin{enumerate}[(a)]
\item $\displaystyle \hspace{5mm} \Pnl(L_1 = k) = \frac{k^{1/2}}{(8\pi)^{1/2}n} e^{-G_\lambda(a)} (1+O(\Cr{M1err})) + O(\Cr{A1err})$,
\item $\displaystyle \hspace{5mm} \Pnl(L_1 \ge k) = \frac{a^{1/2}}{(8\pi)^{1/2}G'_\lambda(a)}e^{-G_\lambda(a)} (1+O(\Cr{M1err})) + O(\Cr{A1err}).$
\end{enumerate}
\end{thm}

Notice in particular that in (b), if $\lambda$ is fixed and $a\to\infty$ then $G'_\lambda(a) = 3a^2/8 + O(a)$, so the correct constant in the denominator of \cite[Proposition 2]{pittel:largest_cpt_rg} should be $(9\pi/8)^{1/2}$ as mentioned above.

Next we give a related result that provides information on the size of the component containing a particular vertex $v$. Let $\Cc(v)$ be the connected component containing $v$, and write $|\Cc(v)|$ for the number of vertices in $\Cc(v)$.

\begin{thm}\label{thm:vLYZ}
Suppose that $|\lambda| \le n^{1/12}$ and $n^{2/3}\le k \le n^{3/4}$. Let $a=k/n^{2/3}$. For any vertex $v$,
\begin{enumerate}[(a)]
\item $\displaystyle \hspace{5mm} \Pnl(|\Cc(v)|=k) = \frac{k^{3/2}}{(8\pi)^{1/2}n^2} e^{-G_\lambda(a)} (1+O(\Cr{M1err}))$,
\item[] \hspace{-7.5mm} and if in addition $\lambda \le \frac15 n^{1/12}$ and $(3\lambda\wedge 1)n^{2/3}\le k$ (as in Theorem \ref{thm:YZ}), then
\item $\displaystyle \hspace{5mm} \Pnl(|\Cc(v)|\ge k) = \frac{a^{3/2}}{(8\pi)^{1/2}n^{1/3} G'_\lambda(a)} e^{-G_\lambda(a)} (1+O(\Cr{M1err})) + O(\Cr{A1err}).$
\end{enumerate}
\end{thm}

Neither of these results has conditions that are strictly best possible; there is nothing special about $n^{3/4}$, for example, and the results could be extended to larger $k$ at some cost to the allowed values of $\lambda$ and the resulting error terms. There is a delicate balance between $n$, $k$, $\lambda$ and the error terms, and the conditions on Theorems \ref{thm:YZ} and \ref{thm:vLYZ} are just one convenient choice in the trade-off between generality and economy. The same applies to other results throughout the paper.

\subsection{Further results and sketch of proof}

In this section we give an outline of our proofs of Theorems \ref{thm:YZ} and \ref{thm:vLYZ}, including some of the results that we develop along the way and which may be useful in their own right. We begin by letting $X(k,k+l)$ be the number of components of our Erd\H{o}s-R\'enyi graph that have exactly $k$ vertices and $k+l$ edges. At the most basic level, our tactic will be to estimate $\Enl[X(k,k+l)]$ very precisely and then show that $\Enl[X(k,k+l)^2]\approx \Enl[X(k,k+l)]$.

Define
\[F_\lambda(x) = x^3/6 - \lambda x^2/2 + \lambda^2 x/2\]
and let $C(k,k+l)$ be the number of possible connected graphs on $k$ vertices with $k+l$ edges. The first step in our proof is to approximate the expected value of $X(k,k+l)$ in terms of $C(k,k+l)$.

\begin{prop}\label{prop:main}
Suppose that $|\lambda|\le n^{1/12}$, $1\le k\le n^{3/4}$ and $l\le 4 n^{1/4}$. Then
\[\Enl[X(k,k+l)] = \frac{C(k,k+l)}{(2\pi)^{1/2} n^l k^{k+1/2}} e^{-F_\lambda(k/n^{2/3})} \left(1+O(\Cl[error]{EXerr})\right)\]
where
\[\Cr{EXerr} = \Cr{EXerr}(k,n,l,\lambda) = \frac{k^4}{n^3} + \frac{|\lambda l|}{n^{1/3}} + \frac{k|\lambda|^3}{n} + \frac{k}{n} + \frac{1}{k}.\]
Furthermore, there exists a finite constant $c$ such that for $|\lambda|\le n^{1/12}$, $1\le k\le n-1$ and $l\ge -1$,
\[\Enl[X(k,k+l)] \le \frac{c}{k}\left(\frac{k^3}{n^2 (l\vee 1)}\right)^{l/2} e^{-F_\lambda(k/n^{2/3})+\frac13 \lambda^3 k n^{-1}} \left(\frac{1+\lambda n^{-1/3}}{1-2n^{-1}}\right)^l.\]
\end{prop}

We will prove Proposition \ref{prop:main} in Section \ref{sec:Xkl}. The quantity $C(k,k+l)$ is interesting in its own right, and has already been the subject of intensive study. See Janson \cite{janson:brownian_wright} for a very useful survey. In Section \ref{sec:wright} we recall and adapt results on the asymptotics of $C(k,k+l)$ which combine with Proposition \ref{prop:main} to give us detailed bounds on $\Enl[X(k,k+l)]$.

We then let $Y(k)$ be the number of components of size exactly $k$, and for $k\le n^{3/4}$, let $Z(k)$ the number of components of size between $k$ and $n^{3/4}$. By summing $X(k,k+l)$ over $l$, in Section \ref{sec:more1mom} we get first moment estimates on $Y(k)$, and then by summing $Y(j)$ over $j$ from $k$ to $n^{3/4}$ we also obtain first moment estimates on $Z(k)$. We could of course sum all the way up to $j=n$, but our bounds are not accurate past a certain point and $n^{3/4}$ is a convenient place to stop. The resulting estimates are below.

\begin{prop}\label{prop:Y}
Suppose that $|\lambda|\le n^{1/12}$ and $n^{2/3}\le k\le n^{3/4}$. Then
\[\Enl[Y(k)] = \frac{k^{1/2}}{(8\pi)^{1/2}n} e^{-G_\lambda(k/n^{2/3})} (1+O(\Cr{M1err})).\]
\end{prop}

\begin{prop}\label{prop:Z}
Suppose that $|\lambda|\le n^{1/12}$ and $(3\lambda\vee 1) n^{2/3}\le k\le n^{3/4}$. Then
\[\Enl[Z(k)] = \frac{1}{(8\pi)^{1/2}} \frac{(k/n^{2/3})^{1/2}}{G'_\lambda(k/n^{2/3})} e^{-G_\lambda(k/n^{2/3})}(1+O(\Cr{M1err})) + O(\Cr{A1err}).\]
\end{prop}

In Section \ref{sec:2mom} we develop second moment bounds on $X(k,k+l)$, $Y(k)$ and $Z(k)$, which follow relatively easily from our first moment calculations thanks to the large amount of independence inherent in Erd\H{o}s-R\'enyi graphs. The final key tool is then developed in Section \ref{sec:explore}: since our combinatorial bounds are weak when the size of the component is very large, we instead use an exploration process as in \cite{nachmias_peres:CRG_mgs} to bound the probability that $L_1>k$, which we will apply with $k = n^{3/4}$.

\begin{lem}\label{le:explore}
If $\lambda \le n^{1/12}/5$, then
\[\Pnl(L_1> k) \le n^{1/4}\exp\Big(-\frac{k^3}{80 n^2}\Big).\]
\end{lem}

This result is designed to be easy to prove rather than accurate, and could be improved by following more closely the argument in \cite{nachmias_peres:CRG_mgs}.

We put the pieces together in Section \ref{sec:momtoprob} by using Markov's inequality and the Cauchy-Schwarz inequality to turn our moment estimates on $Y(k)$ into bounds on the probability that $L_1$ has size exactly $k$; and also to turn our moment estimates on $Z(k)$, together with Lemma \ref{le:explore}, into bounds on the probability that $L_1$ has size at least $k$. These together prove Theorem \ref{thm:YZ}.

In order to prove Theorem \ref{thm:vLYZ} in Section \ref{sec:add}, rather than using second moment bounds, we express the probability that the component containing the vertex $v$ has exactly $k$ vertices and $k+l$ edges directly in terms of $\E[X(k,k+l)]$. We can then apply our first moment calculations to give asymptotics for this probability. Write $|\Cc(v)|$ for the number of vertices in $\Cc(v)$ and $E(\Cc(v))$ for the number of edges in $\Cc(v)$.

\begin{prop}\label{prop:vl}
Suppose that $|\lambda|\le n^{1/12}$, $1\le k\le n^{3/4}$ and $-1\le l\le 4k^{1/2}\wedge 4n^{1/4}$. Then
\begin{align*}
&\Pnl\big(|\Cc(v)|=k,\, E(\Cc(v))=k+l\big)\\
&\hspace{25mm}= \frac{\gamma_l k^{3l/2}}{(2\pi)^{1/2}n^{l+1}} e^{-F_\lambda(k/n^{2/3})}(1+O(\Cr{EXerr} + \tfrac{l^2}{k} + \tfrac{(l+1)^{1/16}}{k^{9/50}}))\\
&\hspace{25mm}= \Big(\frac{3}{8\pi}\Big)^{1/2} \frac{1}{n} \Big(\frac{ek^3}{12n^2(l\vee 1)}\Big)^{l/2} e^{-F_\lambda(k/n^{2/3})} (1+O(\Cr{EXerr} + \tfrac{l^2}{k} + \tfrac{(l+1)^{1/16}}{k^{9/50}} +\tfrac{1}{l})).
\end{align*}
\end{prop}

Summing over $l$ and then $k$, we get the bounds needed for Theorem \ref{thm:vLYZ}. We also give a possibly useful bound on $\P(|\Cc(v)|=k)$ for small $k$.

\begin{prop}\label{prop:vsmallk}
Suppose that $|\lambda|\le n^{1/12}$. For any $M>0$, there exist constants $0<c_1\le c_2<\infty$ depending on $M$ such that
\[c_1 k^{-3/2} e^{-F_\lambda(k/n^{2/3})} \le \Pnl(|\Cc(v)|=k) \le c_2 k^{-3/2} e^{-F_\lambda(k/n^{2/3})}\hspace{5mm} \forall k\le Mn^{2/3}. \]
\end{prop}

\subsection{Related work}

There is a rich literature on Erd\H{o}s-R\'enyi random graphs and related models, beginning with the work of Erd\H{o}s and R\'enyi themselves \cite{erdos_renyi:RG1, erdos_renyi:evolution_RG, erdos_renyi:evolution_RG2}. A good introduction to the general area is provided by the three books \cite{bollobas_book, durrett2007random, janson_et_al:random_graphs}. More directly relevant to this work, besides the articles by \L{}uczak, Pittel and Wierman \cite{luczak_et_al:structure_RG} and Pittel \cite{pittel:largest_cpt_rg}, is a paper by van der Hofstad, Kager and M\"uller \cite{hofstad_et_al:local_limit_CRG} which gives a local limit theorem for the size of the $k$ largest components for arbitrary $k$.

Nachmias and Peres \cite{nachmias_peres:CRG_mgs} give bounds on $\P(L_1\ge an^{2/3})$ via martingale arguments, valid---at least when $\lambda=0$---for any $n>1000$ and $a>8$. Theoretically it should also be possible to extract concrete bounds (that is, specific error bounds and a value of $N$ such that the bounds hold for all $n\ge N$) from our proofs, but we felt it best not to include these in order to keep the article to a reasonable length.

We use combinatorial methods to estimate the distribution of component sizes. Another approach could be to exploit the link between Erd\H{o}s-R\'enyi random graphs and excursions of Brownian motion with parabolic drift, first identified by Aldous \cite{aldous:critical_random_graphs}. Aldous used a breadth-first walk to explore the graph, and showed that the sizes of the largest components, when rescaled by $n^{2/3}$, converge in an appropriate sense to some limit, which he described in detail. This link has since been built upon in various ways. For example, a sharpening of Pittel's result was obtained by van der Hofstad, Janssen and van Leeuwaarden \cite{hofstad_et_al:critical_epidemics}, who also used the same tools to investigate critical SIR epidemics. Addario-Berry, Broutin and Goldschmidt \cite{broutin_et_al:continuum_limit_critical_rgs} showed that in fact the components themselves (rather than just their sizes) converge, when rescaled, to metric spaces characterized by excursions of Brownian motion with parabolic drift decorated by a Poisson point process; they then used this relationship to give various distributional properties of the components \cite{addario_berry_et_al:distn_CRG}.

There are many other random graphs that have similar properties to the near-critical Erd\H{o}s-R\'enyi graph. There is a whole class of models whose component sizes, when suitably rescaled within a critical regime, converge to the lengths of excursions of Brownian motion with parabolic drift just as for the Erd\H{o}s-R\'enyi graph. Some examples include inhomogeneous random graphs \cite{bhamidi_et_al:scaling_inhom_RGs}, the configuration model \cite{dhara_et_al:critical_window_config, joseph:components_critical_RG, riordan:phase_transition_config}, the hierarchical configuration model \cite{hofstad_et_al:mesoscopic_hierarchical}, the quantum random graph \cite{dembo_et_al:component_sizes_quantum_RG} and the $\Delta_{(i)}/G/1$ queue \cite{bet_et_al:heavy_traffic_queues}. This has led to predictions of universality, and it would be very interesting to know whether large deviation results in the style of our Theorems \ref{thm:YZ} and \ref{thm:vLYZ} hold for these related models.

In another direction, O'Connell \cite{oconnell:LDs_RGs} gave a large deviations rate function for the size of the giant component divided by $n$, valid for the Erd\H{o}s-R\'enyi graph $G(n,p)$ with $p=c/n$ for any $c>0$. Bollob\'as and Riordan \cite{bollobas_riordan:asymptotic_normality_RG} showed that the size of the giant component in the regime $p=(1+\eps_n)/n$ with $\eps_n\to 0$ and $\eps_n^3 n\to\infty$ is asymptotically normally distributed when rescaled appropriately. Results similar to those of Pittel \cite{pittel:largest_cpt_rg} have been given by van der Hofstad, Kliem and van Leeuwaarden \cite{hofstad_et_al:cluster_tails_power_law_RGs} for inhomogeneous random graphs whose vertex degrees have power law tails, a model which is \emph{not} in the same universality class as the near-critical Erd\H{o}s-R\'enyi graph. The scaling limit in this case, rather than a Brownian motion with parabolic drift, involves a \emph{thinned L\'evy process}; further related results can be found in \cite{aidekon_et_al:large_devs_thinned_levy}.

Roberts and {\c{S}}eng{\"{u}}l \cite{roberts_sengul:dynamical_ER} consider a dynamical version of the critical ($p=1/n$) Erd\H{o}s-R\'enyi graph in which each edge independently rerandomises itself at rate $1$. That is, each edge has an independent Poisson process of rate $1$, and at each time of the Poisson process it decides---independently of its previous state and all other edges---to be turned on with probability $1/n$ and off with probability $1-1/n$. They have shown that, although at any fixed time the largest component is of order $n^{2/3}$ with high probability, with high probability there exist times in $[0,1]$ when the largest component is much bigger, of order $n^{2/3}\log^{1/3}n$. This is similar to the existence of exceptional times when an infinite component appears in critical dynamical percolation on two-dimensional lattices; see \cite{steif:survey}.

\section{Approximating the expected value of $X(k,k+l)$}\label{sec:Xkl}

Our main aim in this section is to prove Proposition \ref{prop:main}, linking $\Enl[X(k,k+l)]$ to $C(k,k+l)$. We start with the observation
\begin{equation}\label{eq:EX}
\Enl[X(k,k+l)] = \binom{n}{k} C(k,k+l) \pnl^{k+l} (1-\pnl)^{\binom{k}{2}-(k+l)+k(n-k)}.
\end{equation}
To see why this holds, note that there are $\binom{n}{k}$ possible sets of $k$ vertices, and for each choice of $k$ vertices there are $C(k,k+l)$ possible connected graphs on those vertices with exactly $k+l$ edges. For each of these, in order for that graph to appear as a component of our Erd\H{o}s-R\'enyi graph, we must have
\begin{itemize}
\item the specified $k+l$ edges turned on;
\item all other edges between the $k$ chosen vertices turned off;
\item all edges from the $k$ chosen vertices to the $n-k$ other vertices turned off.
\end{itemize}
The probability that this happens is
\[p^{k+l} \times (1-p)^{\binom{k}{2}-(k+l)} \times (1-p)^{k(n-k)}.\]
Therefore \eqref{eq:EX} holds.

To prove Proposition \ref{prop:main} starting from \eqref{eq:EX}, we proceed via a series of simple lemmas approximating the terms $\binom{n}{k}$, $\pnl^{k+l}$ and $(1-\pnl)^{\binom{k}{2} - (k+l)+k(n-k)}$. In our calculations we will often use the expansion
\begin{equation}\label{log}
\log(1+x) = x - \frac{x^2}{2} + \frac{x^3}{3} - \frac{x^4}{4} + \ldots
\end{equation}
to bound $(1+x)^y = e^{y\log(1+x)}$, including as many terms in the expansion as are required to give an accurate estimate. We will also regularly apply Stirling's formula \cite{robbins:remark_stirling},
\[j! = (2\pi)^{1/2} j^{j+1/2} e^{-j} (1+O(1/j)).\]

\begin{lem}\label{le:nchoosek}
If $k^4\le n^3$, we have
\[\binom{n}{k} = \frac{n^k}{(2\pi)^{1/2} k^{k+1/2}}\exp\left(k-\frac{k^2}{2n}-\frac{k^3}{6n^2}\right) \left(1+O(\tfrac{1}{k}+\tfrac{k}{n}+\tfrac{k^4}{n^3})\right).\]
Furthermore, there exists a finite constant $c$ such that for any $k=1,\ldots,n-1$,
\[\binom{n}{k} \le \frac{c n^k}{k^{k+1/2}}\exp\left(k-\frac{k^2}{2n}-\frac{k^3}{6n^2}\right).\]
\end{lem}

\begin{proof}
By Stirling's formula,
\begin{align*}
\binom{n}{k} &= \frac{n^{n+1/2}}{(2\pi)^{1/2} k^{k+1/2}(n-k)^{n-k+1/2}} (1+O(\tfrac{1}{k}+\tfrac{1}{n-k}))\\
&= \frac{n^k}{(2\pi)^{1/2} k^{k+1/2}} \left(1-\frac{k}{n}\right)^{k-n-1/2}(1+O(\tfrac{1}{k}+\tfrac{1}{n-k})).
\end{align*}
We may assume without loss of generality that $n\ge 16$. Then in the case when $k^4\le n^3$ we also have $k\le n/2$, and we can use the expansion \eqref{log} to get
\[\binom{n}{k} = \frac{n^k}{(2\pi)^{1/2} k^{k+1/2}} \left(\exp\left(-\frac{k}{n}-\frac{k^2}{2n^2}-\frac{k^3}{3n^3}+O\Big(\frac{k^4}{n^4}\Big)\right)\right)^{k-n-1/2} (1+O(\tfrac{1}{k})).\]
Simplifying,
\begin{align*}
\binom{n}{k} &= \frac{n^k}{(2\pi)^{1/2} k^{k+1/2}} \exp\left(-\frac{k^2}{n}-\frac{k^3}{2n^2}+k+\frac{k^2}{2n}+\frac{k^3}{3n^2}+O\Big(\frac{k^4}{n^3}\Big) + O\Big(\frac{k}{n}\Big)\right)(1+O(\tfrac{1}{k}))\\
&= \frac{n^k}{(2\pi)^{1/2} k^{k+1/2}} \exp\left(k-\frac{k^2}{2n}-\frac{k^3}{6n^2}\right)\left(1+O(\tfrac{1}{k}+\tfrac{k^4}{n^3}+\tfrac{k}{n})\right).
\end{align*}
This gives the first part of the lemma. For the second part, we note that when we used \eqref{log}, if we had written out the expansion in full we would have obtained
\begin{align*}
\binom{n}{k} &= \frac{n^k}{(2\pi)^{1/2} k^{k+1/2}} \Big(\exp\Big(-\sum_{j=1}^\infty \frac{k^j}{j n^j}\Big)\Big)^{k-n-1/2} (1+O(\tfrac{1}{k}+\tfrac{1}{n-k}))\\
& = \frac{n^k}{(2\pi)^{1/2} k^{k+1/2}} \exp\Big(k-\sum_{j=1}^\infty \frac{k^{j+1}}{j(j+1) n^j}\Big)(1+O(\tfrac{1}{k} + \tfrac{1}{n-k} + \tfrac{k}{n}))\\
& \le \frac{n^k}{(2\pi)^{1/2} k^{k+1/2}} \exp\Big(k-\frac{k^2}{2n}-\frac{k^3}{6n^2}\Big)(1+O(\tfrac{1}{k} + \tfrac{1}{n-k} + \tfrac{k}{n})).
\end{align*}
Since $1\le k\le n-1$, the $O(\cdot)$ terms contribute at most a constant factor.
\end{proof}

\begin{lem}\label{le:p}
Suppose that $|\lambda|\le n^{1/3}/2$, $l |\lambda|\le 4n^{1/3}$ and $k|\lambda|^3 \le 3n$. Then
\[\pnl^{k+l} = \frac{1}{n^{k+l}} e^{\lambda k n^{-1/3} - \frac{1}{2}\lambda^2 k n^{-2/3}}\left(1+ O\Big(\frac{\lambda l}{n^{1/3}}\Big) + O\Big(\frac{k\lambda^3}{n}\Big)\right).\]
Furthermore, for any $\lambda$, $l$, $k$ and $n$, 
\[\pnl^{k+l} \le \frac{1}{n^{k+l}} e^{\lambda k n^{-1/3} - \frac{1}{2}\lambda^2 k n^{-2/3} + \frac{1}{3}\lambda^3 k n^{-1}}(1+\lambda n^{-1/3})^l.\]
\end{lem}

\begin{proof}
First we write
\[p^{k+l} = \frac{1}{n^{k+l}}(1+\lambda n^{-1/3})^{k+l} = \frac{1}{n^{k+l}}(1+\lambda n^{-1/3})^l(1+\lambda n^{-1/3})^k.\]
Using \eqref{log}, if $k|\lambda|^3 \le 3n$ (the value $3$ here is not important---any constant would do---but $3$ will be sufficient when we come to prove Proposition \ref{prop:main}) then
\[(1+\lambda n^{-1/3})^k = e^{k(\lambda n^{-1/3} - \lambda^2 n^{-2/3}/2 + O(\lambda^3 n^{-1}))} = e^{k(\lambda n^{-1/3} - \lambda^2 n^{-2/3}/2)}\left(1+O\Big(\frac{k\lambda^3}{n}\Big)\right)\]
and similarly if $l |\lambda|\le n^{1/3}$ then
\[(1+\lambda n^{-1/3})^l = 1 + O(\lambda l n^{-1/3}).\]
This gives the first part of the lemma; for the second, we use instead the fact that $\log(1+x)\le x-x^2/2+x^3/3$ for all $x$.
\end{proof}

\begin{lem}\label{le:1mp}
Suppose that $|\lambda|\le n^{1/3}/2$. If $k + l \le 3n$ then
\[(1-\pnl)^{\binom{k}{2}-(k+l) + k(n-k)} = e^{-k - \lambda k n^{-1/3} + k^2/(2n) + \lambda k^2 / (2n^{4/3})} \left(1 + O(\tfrac{k+l}{n})\right).\]
Furthermore, there exists a constant $c$ such that whenever $|\lambda|\le n^{1/3}/2$,
\[(1-\pnl)^{\binom{k}{2}-(k+l) + k(n-k)} \le  c e^{-k - \lambda k n^{-1/3} + k^2/(2n) + \lambda k^2 / (2n^{4/3})}(1-2/n)^{-l}.\]
\end{lem}

\begin{proof}
We start by writing
\[(1-\pnl)^{\binom{k}{2}-(k+l) + k(n-k)} = (1-\pnl)^{kn-k^2/2}(1-\pnl)^{-3k/2-l}\]
and treating the two terms separately. For the last term, if $k+l\le 3n$ then
\[(1-\pnl)^{-3k/2-l} = 1+O(\tfrac{k+l}{n}),\]
and otherwise (since we always have $k\le n$)
\[(1-\pnl)^{-3k/2-l} = (1-\pnl)^{-l}(1+O(k/n)) \le c(1-2/n)^{-l}.\]
For the other term $(1-\pnl)^{kn-k^2/2}$, we use \eqref{log} to get
\begin{align*}
(1-\pnl)^{kn-k^2/2} &= e^{-kn\pnl + k^2\pnl/2}(1+O(kn\pnl^2)+O(k^2\pnl^2))\\
&= e^{-k - \lambda k n^{-1/3} + k^2/(2n) + \lambda k^2/(2n^{4/3})}(1+O(\tfrac{k}{n})).
\end{align*}
Combining these estimates gives the result.
\end{proof}

We can now complete the proof of Proposition \ref{prop:main} by combining the results above.

\begin{proof}[Proof of Proposition \ref{prop:main}]
Our starting point is \eqref{eq:EX}, which said that
\[\Enl[X(k,k+l)] = \binom{n}{k} C(k,k+l) \pnl^{k+l} (1-\pnl)^{\binom{k}{2}-(k+l)+k(n-k)}.\]
We first concentrate on the first bound. For $k\le n^{3/4}$, the first part of Lemma \ref{le:nchoosek} holds, i.e.
\[\binom{n}{k} = \frac{n^k}{(2\pi)^{1/2} k^{k+1/2}}\exp\left(k-\frac{k^2}{2n}-\frac{k^3}{6n^2}\right) \left(1+O(\tfrac{1}{k}+\tfrac{k}{n}+\tfrac{k^4}{n^3})\right).\]
Also, when $k\le n^{3/4}$, $l\le 4 n^{1/4}$ and $|\lambda|\le n^{1/12}$, the first part of Lemma \ref{le:p} holds, i.e.
\[\pnl^{k+l} = \frac{1}{n^{k+l}} e^{\lambda k n^{-1/3} - \frac{1}{2}\lambda^2 k n^{-2/3}}\left(1+ O\Big(\frac{\lambda l}{n^{1/3}}\Big) + O\Big(\frac{k\lambda^3}{n}\Big)\right).\]
Thirdly, when $k\le n^{3/4}$ and $l\le 4n^{1/4}$, the first part of Lemma \ref{le:1mp} holds, i.e.
\[(1-\pnl)^{\binom{k}{2}-(k+l) + k(n-k)} = e^{-k - \lambda k n^{-1/3} + k^2/(2n) + \lambda k^2 / (2n^{4/3})} \left(1 + O(\tfrac{k+l}{n})\right).\]
Combining the three approximations above gives the first part of the result, except with an extra error term $O(l/n)$. Since $l\le 4n^{1/4}$, we have $l/n \le 4n^{-3/4}$; and by optimising over $k$, the terms $k^4/n^3$ and $1/k$ in $\Cr{EXerr}$ are at least of order $n^{-3/5}$, which is larger than $n^{-3/4}$. Therefore we can absorb the $O(l/n)$ term into $\Cr{EXerr}$ and the first part of the result follows.

For the second part of the result, we instead use the second parts of Lemmas \ref{le:nchoosek}, \ref{le:p} and \ref{le:1mp}, together with the approximation
\[C(k,k+l)\le c (l\vee 1)^{-l/2}k^{k+(3l-1)/2}\]
which holds for all $l$ and $k$ and is due to Bollob\'as \cite[Corollary 5.21]{bollobas_book}.
\end{proof}

\section{Wright's coefficients}\label{sec:wright}

We can see from Proposition \ref{prop:main} that we need to know how $C(k,k+l)$ behaves. Cayley's formula tells us that $C(k,k-1)=k^{k-2}$. It is also well-known that $C(k,k) = (\pi/8)^{1/2}k^{k-1/2}(1+O(k^{-1/2}))$; see \cite[Corollary 5.19]{bollobas_book} for example. In this section we give details of asymptotics for $C(k,k+l)$ for other values of $l$. We develop only the estimates that we need to prove Theorems \ref{thm:YZ} and \ref{thm:vLYZ}, and the reader who is interested in a more thorough treatment of $C(k,k+l)$ as an interesting object in itself may wish to refer to Janson \cite{janson:brownian_wright}.

Wright's coefficients ($\gamma_l$, $l\ge 1$) satisfy
\begin{equation}\label{eq:gamma}
\gamma_l = \frac{\pi^{1/2}3^l (l-1)! d_l}{2^{(5l-1)/2} \Gamma(3l/2)}
\end{equation}
where $d_1=d_2=5/36$ and
\[d_{l+1} = d_l + \sum_{i=1}^{l-1} \frac{d_i d_{l-i}}{(l+1)\binom{l}{i}}, \,\, l\ge 2.\]
Wright \cite{wright:connectedIII} gave an asymptotic for $C(k,k+l)$ in terms of $\gamma_l$ for $l=o(k^{1/3})$. This was later improved by various authors, and in particular we now know that
\begin{equation}\label{eq:C}
C(k,k+l) = \gamma_l k^{k+(3l-1)/2} (1+O(\tfrac{l^{2}}{k}+\tfrac{(l+1)^{1/16}}{k^{9/50}})) \hspace{4mm} \hbox{ for all } l\le 4k^{1/2};
\end{equation}
see \cite{bender_et_al:asymptotic_connected_graphs}. See also \cite{luczak:sparse_connected_graphs} and \cite{hofstad_spencer:connected_graphs} for two beautiful proofs of slightly less precise asymptotics, the former only a few pages long and using the Erd\H{o}s-R\'enyi random graph.

Clearly the sequence $(d_l)$ is increasing; Wright \cite[Theorem 3]{wright:connectedIII} showed that it is bounded above and therefore converges to some limit $d$, which Voblyi \cite[Theorem 3]{voblyi:wright} identified as $1/(2\pi)$. We can adapt Wright's proof that the sequence converges to bound the rate of convergence.

\begin{lem}\label{le:dconv}
As $l\to\infty$, $1/(2\pi)-d_l = O(1/l)$.
\end{lem}

\begin{proof}
Since $d_l$ is an increasing sequence, $d_l\le d=1/(2\pi)$ for all $l$. Therefore for any $j\ge 2$,
\begin{align*}
d_{j+1}-d_j = \sum_{i=1}^{j-1} \frac{d_id_{j-i}}{(j+1)\binom{j}{i}} &\le 2d^2 \sum_{i=1}^{\lfloor j/2\rfloor} \frac{1}{(j+1)\binom{j}{i}}\\
&\le 2d^2 \sum_{i=1}^{\lfloor j/2\rfloor} \frac{i}{(j+1)j} \frac{(i-1)(i-2)\ldots 1}{(j-1)(j-2)\ldots (j-i+1)}\\
&\le \frac{2d^2}{j(j+1)} \sum_{i=1}^{\lfloor j/2\rfloor} i \left(\frac{i}{j}\right)^{i-1} \le \frac{2d^2}{j(j+1)} \sum_{i=1}^{\infty} i \left(\frac{1}{2}\right)^{i-1} = \frac{8d^2}{j(j+1)}.
\end{align*}
Thus
\[1/(2\pi)-d_l = d-d_l = \sum_{j=l}^\infty (d_{j+1}-d_j) \le 8d^2 \sum_{j=l}^\infty 1/j^2 = O(1/l).\qedhere\]
\end{proof}

\noindent
Applying Stirling's formula and then Lemma \ref{le:dconv} to \eqref{eq:gamma}, we obtain for $l\ge 1$
\begin{equation}\label{eq:gammaasymp}
\gamma_l = d_l (3\pi)^{1/2} \left(\frac{e}{12 l}\right)^{l/2}(1+O(1/l)) = \Big(\frac{3}{4\pi}\Big)^{1/2} \left(\frac{e}{12 l}\right)^{l/2}(1+O(1/l)).
\end{equation}
We also know that $C(k,k-1)=k^{k-2}$ and $C(k,k) = (\pi/8)^{1/2}k^{k-1/2}(1+O(k^{-1/2}))$, so in fact if we replace $\frac{e}{12l}$ with $\frac{e}{12(l\vee 1)}$ then \eqref{eq:gammaasymp} holds for all $l\ge -1$.

Combining this with Proposition \ref{prop:main} and \eqref{eq:C}, we immediately get the following corollary. The first line of the equality is more useful for small $l$, and the second line for larger $l$.

\begin{cor}\label{cor:EX}
Suppose that $|\lambda|\le n^{1/12}$, $1\le k\le n^{3/4}$ and $-1\le l\le 4k^{1/2}\wedge 4n^{1/4}$. Then
\begin{align*}
\Enl[X(k,k+l)] &= \frac{\gamma_l k^{3l/2-1}}{(2\pi)^{1/2}n^l} e^{-F_\lambda(k/n^{2/3})}(1+O(\Cr{EXerr} + \tfrac{l^2}{k} + \tfrac{(l+1)^{1/16}}{k^{9/50}}))\\
&= \Big(\frac{3}{8\pi}\Big)^{1/2} \frac{1}{k} \Big(\frac{ek^3}{12n^2(l\vee 1)}\Big)^{l/2} e^{-F_\lambda(k/n^{2/3})} (1+O(\Cr{EXerr} + \tfrac{l^2}{k} + \tfrac{(l+1)^{1/16}}{k^{9/50}} +\tfrac{1}{l})).
\end{align*}
\end{cor}

The following bound will be useful when $k$ is small.

\begin{cor}\label{cor:smallk}
For any $M>0$, there exist constants $0<c_1\le c_2<\infty$ depending on $M$ such that
\[c_1 k^{k-2}n \le \sum_{l=-1}^L \frac{C(k,k+l)}{n^l} \le c_2 k^{k-2}n\]
for all $n\ge 1$, $k\le Mn^{2/3}$, and $L\le 4k^{1/2}$.
\end{cor}

\begin{proof}
By \eqref{eq:C} and \eqref{eq:gammaasymp}, we have
\[c_1 \Big(\frac{e}{12(l\vee 1)}\Big)^{l/2} k^{k+(3l-1)/2} \le C(k,k+l) \le c_2 \Big(\frac{e}{12(l\vee 1)}\Big)^{l/2} k^{k+(3l-1)/2}\]
for all $-1\le l\le 4k^{1/2}$. Dividing by $n^l$ and summing over $l\le L$,
\[c_1 k^{k-1/2} \sum_{l=-1}^L \Big(\frac{ek^3}{12n^2(l\vee 1)}\Big)^{l/2} \le \sum_{l=-1}^L \frac{C(k,k+l)}{n^l} \le c_2 k^{k-1/2} \sum_{l=-1}^L \Big(\frac{ek^3}{12n^2(l\vee 1)}\Big)^{l/2}.\]
Since $k^3/n^2\le M^3$, the sum is bounded above and below by constants times its first term, which gives the result.
\end{proof}

\noindent
Inspecting the second line of Corollary \ref{cor:EX}, and letting $a=k/n^{2/3}$, we see that if we are interested in all components of size $k$ when $k\approx an^{2/3}$ then we will need to estimate $\sum_l (\frac{e a^3}{12 l})^{l/2}$ for large $a$. The following lemma will be useful for this purpose.

\begin{lem}\label{le:opt}
If $x\in \left[\frac{a^3}{12}(1-a^{-4/3}), \frac{a^3}{12}(1+a^{-4/3})\right]$, then
\[\left(\frac{e a^3}{12x}\right)^{x/2} = \exp\left(\frac{a^3}{24} - \frac{a^3y^2}{48} + O(a^{-1})\right)\]
where $y=12x/a^3 -1$.
\end{lem}

\begin{proof}
Let $y = 12x/a^3 -1$, so that $x = \frac{a^3}{12}(1+y)$. Then using \eqref{log},
\begin{align*}
\left(\frac{e a^3}{12x}\right)^{x/2} = \left(\frac{e}{1+y}\right)^{a^3(1+y)/24} &= e^{a^3(1+y)/24 - a^3(1+y)\log(1+y)/24}\\
&= e^{a^3(1+y)/24 - a^3(1+y)(y-y^2/2+O(y^3))/24}\\
&= e^{a^3/24 - a^3 y^2/48 +O(a^3 y^3)}.
\end{align*}
Now, $|y|\le a^{-4/3}$, so $|a^3 y^3| \le 1/a$ and the result follows.
\end{proof}

This allows us to bound the sum of the terms involving $l$ in Proposition \ref{prop:main}.

\begin{lem}\label{le:lsumasymp}
If $k\ge n^{2/3}$ and $k^3/n^2 \le L \le 4k^{1/2}$, then
\[\sum_{l=-1}^L \frac{C(k,k+l)}{n^l} = \frac{k^{k+1}}{2n} \exp\Big(\frac{k^3}{24 n^2}\Big) (1+O(\tfrac{n^{2/3}}{k}+\tfrac{k^5}{n^4}+\tfrac{k^{3/400}}{n^{1/8}})).\]
\end{lem}

\begin{proof}
Write $a=k/n^{2/3}$. When $a$ is a constant, the $O(n^{2/3}/k)$ error means that we only need to bound above and below by constants times $k^{k-2}/n$, which we did in Corollary \ref{cor:smallk} (in this case $k^{k+1}/n$ and $k^{k-2}n$ are of the same order). Therefore we may assume that $a$ is large.

Note that $(\frac{e a^3}{12l})^{l/2}$ is increasing in $l$ for $l\le a^3/12$, and decreasing in $l$ for $l\ge a^3/12$. Let $J^- = \lfloor \frac{a^3}{12}(1-a^{-4/3})\rfloor$ and $J^+ = \lfloor \frac{a^3}{12}(1+a^{-4/3})\rfloor$. Clearly
\[\sum_{l=J^-}^{J^+} \frac{C(k,k+l)}{n^l} \le \sum_{l=-1}^{L} \frac{C(k,k+l)}{n^l} \le \sum_{l=J^-}^{J^+} \frac{C(k,k+l)}{n^l} + \sum_{l=-1}^{J^-} \frac{C(k,k+l)}{n^l} + \sum_{l=J^+}^\infty \frac{C(k,k+l)}{n^l}.\]
We start by estimating $\sum_{l=J^-}^{J^+} \frac{C(k,k+l)}{n^l}$, and then show that the other two terms on the right-hand side are small.

By \eqref{eq:C} and \eqref{eq:gammaasymp}, for $l\in [J^-,J^+]$,
\[C(k,k+l) = \Big(\frac{3}{4\pi}\Big)^{1/2} k^{k-1/2} \Big(\frac{ek^3}{12 l}\Big)^{l/2} (1+O(\tfrac{n^2}{k^3}+\tfrac{k^5}{n^4}+\tfrac{k^{3/400}}{n^{1/8}})).\]
Thus
\begin{equation}\label{eq:sumJ}
\sum_{l=J^-}^{J^+} \frac{C(k,k+l)}{n^l} = (1+O(\tfrac{n^2}{k^3}+\tfrac{k^5}{n^4}+\tfrac{k^{3/400}}{n^{1/8}}))\Big(\frac{3}{4\pi}\Big)^{1/2} k^{k-1/2} \sum_{l=J^-}^{J^+} \Big(\frac{e a^3}{12 l}\Big)^{l/2}.
\end{equation}

Now, note that $(\frac{e a^3}{12x})^{x/2}$ is increasing for $x<a^3/12$ and decreasing for $x>a^3/12$. By considering these two regions separately, and using the facts that
\[\int_{l-1}^l f(x)dx \le f(l) \le \int_l^{l+1} f(x) dx \hspace{2mm} \hbox{ if $f$ is increasing on } (l-1,l+1)\]
and
\[\int_{l}^{l+1} f(x)dx \le f(l) \le \int_{l-1}^{l} f(x) dx \hspace{2mm} \hbox{ if $f$ is decreasing on } (l-1,l+1)\]
we have
\[\int_{J^-+1}^{J^+-1} \left(\frac{e a^3}{12x}\right)^{x/2} dx - 2 e^{a^3/24} \le \sum_{l=J^-+1}^{J^+-1} \left(\frac{e a^3}{12l}\right)^{l/2} \le \int_{J^-}^{J^+} \left(\frac{e a^3}{12x}\right)^{x/2} dx + 2 e^{a^3/24}.\]
But by Lemma \ref{le:opt},
\begin{align*}
\int_{J^-}^{J^+} \left(\frac{e a^3}{12x}\right)^{x/2} dx &= (1+O(a^{-1}))\frac{a^3}{12} \int_{12 J^-/a^3-1}^{12J^+/a^3-1} e^{a^3/24-a^3 y^2/48} dy\\
&= (1+O(a^{-1}))\frac{a^3}{12} e^{a^3/24} \int_{-\infty}^{\infty} e^{-a^3 y^2/48} dy\\
&= (1+O(a^{-1}))\frac{a^3}{12} e^{a^3/24} (48\pi/a^3)^{1/2}
\end{align*}
and similarly for $\int_{J^-+1}^{J^+-1} (\frac{e a^3}{12x})^{x/2} dx$ (using the assumption, from the beginning of the proof, that $a$ is large). Plugging this estimate into \eqref{eq:sumJ} gives
\[\sum_{l=J^-}^{J^+} \frac{C(k,k+l)}{n^l} = \frac{k^{k+1}}{2n} \exp\Big(\frac{k^3}{24 n^2}\Big) (1+O(\tfrac{n^{2/3}}{k}+\tfrac{k^5}{n^4}+\tfrac{k^{3/400}}{n^{1/8}})),\]
so it now suffices to show that both $\sum_{l=-1}^{J^-} \frac{C(k,k+l)}{n^l}$ and $\sum_{l=J^+}^\infty \frac{C(k,k+l)}{n^l}$ are relatively small.

By Lemma \ref{le:opt},
\[\sum_{l=1}^{J^-} \left(\frac{ea^3}{12 l}\right)^{l/2} \le \frac{a^3}{12} \exp\left(\frac{a^3}{24} - \frac{a^{1/3}}{48} + O(a^{-1})\right).\]
Similarly,
\[\sum_{l= J^+}^{\lfloor a^3\rfloor} \left(\frac{ea^3}{12 l}\right)^{l/2} \le a^3 \exp\left(\frac{a^3}{24} - \frac{a^{1/3}}{48} + O(a^{-1})\right),\]
and trivially
\[\sum_{l=\lfloor a^3\rfloor}^\infty \left(\frac{e}{12}\right)^{l/2} \le \frac{1}{2}e^{-a^3/2}.\]
Combining these bounds with \eqref{eq:C} and \eqref{eq:gammaasymp}, we get the result.
\end{proof}

\section{More first moment asymptotics}\label{sec:more1mom}

In this section our main aim is to prove Propositions \ref{prop:Y} and \ref{prop:Z}. For $k\in\mathbb{N}$, we recall that $Y(k)$ is the number of components of size exactly $k$, and for $k\le n^{3/4}$, $Z(k)$ is the number of components of size between $k$ and $n^{3/4}$. Propositions \ref{prop:Y} and \ref{prop:Z} give first moment asymptotics for $Y(k)$ and $Z(k)$ respectively. 

To prove Proposition \ref{prop:Y}, we write
\[\Enl[Y(k)]=\sum_{l\ge -1} \Enl[X(k,k+l)].\]
We will use Proposition \ref{prop:main} and Lemma \ref{le:lsumasymp} to bound the sum for small $l$, but first we need a bound on the sum for large $l$.

\begin{lem}\label{le:lbig}
There exists a finite constant $c$ such that if $n\ge 25$, $|\lambda|\le n^{1/12}$, $1\le k\le n^{3/4}$ and $L\ge 4n^{1/4}\wedge 4k^{1/3}$,
\[\sum_{l=L}^\infty \Enl[X(k,k+l)] \le c  e^{-L/2}.\]
\end{lem}

\begin{proof}
First we claim that
\[\sum_{l=L}^\infty \Big(\frac{k^3}{n^2 l}\Big)^{l/2} \Big(\frac{1+\lambda n^{-1/3}}{1-2n^{-1}}\Big)^l \le  3e^{-L/2}.\]
It is easy to check that when $\lambda\le n^{1/12}$ and $n\ge 25$, we have $\frac{1+\lambda n^{-1/3}}{1-2n^{-1}} \le 2e^{-1/2}$; and also that when $l\ge L$ and $k\le n^{3/4}$, we have $k^3/(n^2 l) \le 1/4$. Therefore
\[\sum_{l=L}^\infty \Big(\frac{k^3}{n^2 l}\Big)^{l/2} \Big(\frac{1+\lambda n^{-1/3}}{1-2n^{-1}}\Big)^l \le \sum_{l=L}^\infty e^{-l/2} \le 3e^{-L/2}\]
as claimed. But the second part of Proposition \ref{prop:main} tells us that
\[\Enl[X(k,k+l)] \le \frac{c}{k} \Big(\frac{k^3}{n^2 l}\Big)^{l/2} e^{-F_\lambda(k/n^{2/3})+\frac13 \lambda^3 k n^{-1}} \Big(\frac{1+\lambda n^{-1/3}}{1-2n^{-1}}\Big)^l.\]
It is easy to check that $F_\lambda(x)\ge0$ for all $x\ge 0$, and since $|\lambda|\le 1/12$ and $k\le n^{3/4}$ we have $\lambda^3 k n^{-1} \le 1$ so we get the result.
\end{proof}

\begin{proof}[Proof of Proposition \ref{prop:Y}]
Proposition \ref{prop:main} tells us that for $l\le 4n^{1/4}$,
\[\Enl[X(k,k+l)] = \frac{C(k,k+l)}{(2\pi)^{1/2} n^l k^{k+1/2}} e^{-F_\lambda(k/n^{2/3})} (1+O(\Cr{EXerr}(k,n,l,\lambda))).\]
Applying Lemma \ref{le:lsumasymp} with $L=4n^{1/4}$ (which is larger than $k^3/n^2$ since $k\le n^{3/4}$) gives
\[\sum_{l=-1}^{\lfloor 4 n^{1/4}\rfloor} \frac{C(k,k+l)}{n^l} = \frac{k^{k+1}}{2n} e^{k^3/(24n^2)} (1+O(\tfrac{n^{2/3}}{k} + \tfrac{k^5}{n^4}+\tfrac{k^{3/400}}{n^{1/8}})).\]
Combining these and simplifying the error terms, we get
\[\sum_{l=-1}^{\lfloor 4 n^{1/4}\rfloor} \Enl[X(k,k+l)] = \frac{k^{1/2}}{(8\pi)^{1/2} n} e^{-G_\lambda(k/n^{2/3})}(1+O(\Cr{M1err})).\]
Finally, Lemma \ref{le:lbig} gives that for large $n$,
\[\sum_{l=\lceil 4n^{1/4}\rceil}^\infty \Enl[X(k,k+l)] \le c e^{-2n^{1/4}}\]
which can be absorbed into the $O(\Cr{M1err})$ term.
\end{proof}

We now move on to proving Proposition \ref{prop:Z}. We will sum $\Enl[Y(j)]$ over $j$ from $k$ to $\lfloor n^{3/4}\rfloor$; we simply need to estimate the resulting sum, which we do in the next two lemmas.

\begin{lem}\label{le:Zcalcint}
For $a\ge 1\vee 3\lambda$ and any constant $r\ge 0$, 
\[\int_a^\infty y^{r} e^{-G_\lambda(y)} dy = \frac{a^{r}}{G_\lambda'(a)}e^{-G_\lambda(a)}(1 + O(\tfrac{1}{a^3})).\]
\end{lem}

\begin{proof}
Writing $y^{r} e^{-G_\lambda(y)}$ as $(y^{r}/G'_\lambda(y))\cdot G'_\lambda(y) e^{-G_\lambda(y)}$ and integrating by parts, we get
\[\int_a^\infty y^{r} e^{-G_\lambda(y)} dy = \frac{a^{r}}{G'_\lambda(a)}e^{-G_\lambda(a)} + \int_a^\infty \Big(\frac{r}{y G'_\lambda(y)} - \frac{G_\lambda''(y)}{G_\lambda'(y)^2}\Big) y^{r}e^{-G_\lambda(y)} dy.\]
Note that $G'_\lambda(y) = 3y^2/8 - \lambda y + \lambda^2/2$, which is strictly positive for $y>2\lambda$, and moreover if $y\ge a$ and $\lambda \le a/3$, we have $1/G'_\lambda(a) = O(1/a^2)$. It is then straightforward to check that under the same conditions we have
\[\frac{r}{y G'_\lambda(y)} - \frac{G_\lambda''(y)}{G_\lambda'(y)^2} = O\Big(\frac{1}{a^3}\Big).\]
Thus
\[\Big(1+O\Big(\frac{1}{a^3}\Big)\Big) \int_a^\infty y^{r} e^{-G_\lambda(y)} dy = \frac{a^{r}}{G'_\lambda(a)}e^{-G_\lambda(a)}\]
and the result follows.
\end{proof}

\begin{lem}\label{le:Zsumtoint}
For $(1\vee 3\lambda) n^{2/3}\le k \le n^{3/4}$ and $|\lambda|\le n^{1/12}$, and any constant $r\ge0$,
\[\frac{1}{n^{2/3}}\sum_{j=k}^{\lfloor n^{3/4}\rfloor} \Big(\frac{j}{n^{2/3}}\Big)^{r} e^{-G_\lambda(j/n^{2/3})} = \frac{(k/n^{2/3})^{r}}{G'_\lambda(k/n^{2/3})} e^{-G_\lambda(k/n^{2/3})} (1+O(\tfrac{n^2}{k^3})) + O(n^{(r-2)/12}e^{-G_\lambda(n^{1/12})}).\]
\end{lem}

\begin{proof}
Write $N = \lfloor n^{3/4}\rfloor$ and
\[S=\frac{1}{n^{2/3}}\sum_{j=k}^N \Big(\frac{j}{n^{2/3}}\Big)^{r} e^{-G_\lambda(j/n^{2/3})}.\]
For a lower bound, since $x^{r}$ is increasing in $x$ and $e^{-G_\lambda(x/n^{2/3})}$ is decreasing in $x$ for $x\ge 2\lambda n^{2/3}$, for any $j\ge k$ we have
\[\Big(\frac{j}{n^{2/3}}\Big)^{r} e^{-G_\lambda(j/n^{2/3})} \ge \Big(\frac{x-1}{n^{2/3}}\Big)^{r} e^{-G_\lambda(x/n^{2/3})} \hspace{3mm} \forall x\in[j,j+1)\]
and therefore
\[S \ge \frac{1}{n^{2/3}}\int_k^N \Big(\frac{x-1}{n^{2/3}}\Big)^{r} e^{-G_\lambda(x/n^{2/3})} dx.\]
Now $(x-1)^{r} = x^{r}(1 + O(1/x))$, so
\[S \ge (1+O(\tfrac{1}{k}))\frac{1}{n^{2/3}}\int_k^N \Big(\frac{x}{n^{2/3}}\Big)^{r} e^{-G_\lambda(x/n^{2/3})} dx.\]
Substituting $y=x/n^{2/3}$, writing $a=k/n^{2/3}$ and applying Lemma \ref{le:Zcalcint}, we get
\[S \ge \frac{a^{r}}{G_\lambda'(a)} e^{-G_\lambda(a)}(1+O(\tfrac{1}{a^3}+\tfrac{1}{k})) - \frac{n^{r/12}}{G_\lambda'(n^{1/12})} e^{-G_\lambda(n^{1/12})}(1+O(1)).\]
Noting that $G'_\lambda(x)$ is of order $x^2$ for all $x\ge 1\vee 3\lambda$ gives the desired lower bound ($1/k\le 1/a^3$ so we may ignore that error term). We proceed similarly for the upper bound. Again since $x^{r}$ is increasing in $x$ and $e^{-G_\lambda(x/n^{2/3})}$ is decreasing in $x$ for $x\ge 2\lambda n^{2/3}$, for any $j\ge k$ we have
\[\Big(\frac{j}{n^{2/3}}\Big)^{r} e^{-G_\lambda(j/n^{2/3})} \le \Big(\frac{x}{n^{2/3}}\Big)^{r} e^{-G_\lambda((x-1)/n^{2/3})} \hspace{3mm} \forall x\in[j,j+1)\]
and therefore
\begin{align*}
S &\le \frac{1}{n^{2/3}}\Big(\frac{k}{n^{2/3}}\Big)^{r} e^{-G_\lambda(k/n^{2/3})} + \frac{1}{n^{2/3}}\int_{k+1}^{N+1} \Big(\frac{x}{n^{2/3}}\Big)^{r} e^{-G_\lambda((x-1)/n^{2/3})} dx\\
&= \frac{1}{n^{2/3}}\Big(\frac{k}{n^{2/3}}\Big)^{r} e^{-G_\lambda(k/n^{2/3})} + \frac{1}{n^{2/3}}\int_{k}^{N} \Big(\frac{x+1}{n^{2/3}}\Big)^{r} e^{-G_\lambda(x/n^{2/3})} dx\\
&= \frac{1}{n^{2/3}}\Big(\frac{k}{n^{2/3}}\Big)^{r} e^{-G_\lambda(k/n^{2/3})} + (1+O(\tfrac{1}{k}))\frac{1}{n^{2/3}}\int_{k}^{N} \Big(\frac{x}{n^{2/3}}\Big)^{r} e^{-G_\lambda(x/n^{2/3})} dx.
\end{align*}
Now again substituting $y=x/n^{2/3}$, writing $a=k/n^{2/3}$ and applying Lemma \ref{le:Zcalcint}, we have
\[\frac{1}{n^{2/3}}\int_{k}^{N} \Big(\frac{x}{n^{2/3}}\Big)^{r} e^{-G_\lambda(x/n^{2/3})} dx = \frac{a^{r}}{G'_\lambda(a)} e^{-G_\lambda(a)} (1+O(\tfrac{1}{a^3})) - \frac{n^{r/12}}{G_\lambda'(n^{1/12})} e^{-G_\lambda(n^{1/12})}(1+O(1)).\]
Again we use the fact that $G'_\lambda(x)$ is of order $x^2$ for $x \ge 1\vee 3\lambda$, which means that both
\[\frac{n^{r/12}}{G'_\lambda(n^{1/12})} = O(n^{(r-2)/12}) \hspace{5mm} \hbox{ and } \hspace{5mm} \frac{1}{n^{2/3}}= O\Big(\frac{k^2}{n^3}\Big)\frac{1}{G'_\lambda(a)}.\]
These estimates give us
\[S \le O\Big(\frac{k^2}{n^3}\Big)\frac{a^{r}}{G'_\lambda(a)} e^{-G_\lambda(a)} + \frac{a^{r}}{G'_\lambda(a)} e^{-G_\lambda(a)}(1+O(\tfrac{1}{a^3})) + O(n^{(r-2)/12}e^{-G_\lambda(n^{1/12})}).\]
Noting that $k^2/n^3 \le 1/a^3$, we can drop that error term and the proof is complete.
\end{proof}

\begin{proof}[Proof of Proposition \ref{prop:Z}]
Recall that $Z(k)$ is the number of components of size between $k$ and $n^{3/4}$. By Proposition \ref{prop:Y},
\[\Enl[Z(k)] = \sum_{j=k}^{\lfloor n^{3/4}\rfloor} \Enl[Y(j)] = \sum_{j=k}^{\lfloor n^{3/4}\rfloor} \frac{j^{1/2}}{(8\pi)^{1/2}n} e^{-G_\lambda(j/n^{2/3})}(1+O(\Cr{M1err}(j,n,\lambda))).\]
By Lemma \ref{le:Zsumtoint} (carefully summing the terms involving $j$ in $\Cr{M1err}(j,n,\lambda)$ by changing the value of $r$ used in Lemma \ref{le:Zsumtoint}), this is
\[\frac{(k/n^{2/3})^{1/2}}{G'_\lambda(k/n^{2/3})} e^{-G_\lambda(k/n^{2/3})} (1+O(\Cr{M1err}(k,n,\lambda))) + O(n^{-1/8}e^{-G_\lambda(n^{1/12})}).\]
It is easily checked that since $\lambda \le n^{1/12}/3$, we have $G_\lambda(n^{1/12}) \ge n^{1/4}/72$ so the last error is of order smaller than $\Cr{A1err}$.
\end{proof}

\section{Second moment bounds}\label{sec:2mom}

We now have the bounds that we want on $\Enl[X(k,k+l)]$, $\Enl[Y(k)]$ and $\Enl[Z(k)]$. To show that these expectations give asymptotically tight bounds on $\Pnl(X(k,k+l)\ge 1)$, $\Pnl(Y(k)\ge 1)$, and $\Pnl(Z(k)\ge 1)$, we need to develop some second moment bounds.

\begin{lem}\label{le:EX2toEX1}
If $k\le n^{3/4}$ then
\[\Enl[X(k,k+l)^2] = \Enl[X(k,k+l)] + \Enl[X(k,k+l)]^2 \exp\left(\frac{\lambda k^2}{n^{4/3}} - \frac{k^3}{n^2}\right) (1+O(\tfrac{k}{n}+\tfrac{k^4}{n^3})).\]
Also, if $j,k\le n^{3/4}$, $n\ge 81$ and either $j\neq k$ or $l\neq l'$, then
\begin{multline*}
\Enl[X(j,j+l)X(k,k+l')]\\
= \Enl[X(j,j+l)]\Enl[X(k,k+l')] \exp\left(\frac{\lambda jk}{n^{4/3}} - \frac{jk}{2n^2}(j+k)\right) (1+O(\tfrac{j+k}{n}+\tfrac{j^4+k^4}{n^3})).
\end{multline*}
\end{lem}

\begin{proof}
Let $\mathcal C(k,k+l)$ be the set of components with $k$ vertices and $k+l$ edges. First note that
\[\Enl[X(k,k+l)^2] = \Enl\bigg[\sum_{|S|=k} \sum_{|S'|=k} \ind_{\{S\in \mathcal C(k,k+l)\}}\ind_{\{S'\in \mathcal C(k,k+l)\}}\bigg].\]
If $S$ is a component, then $S$ and $S^c$ are not connected by an edge; therefore if $S$ and $S'$ are components with $S\cap S'\neq\emptyset$, then necessarily $S=S'$. Thus
\[\Enl[X(k,k+l)^2] = \Enl\bigg[\sum_{|S|=k} \ind_{\{S\in \mathcal C(k,k+l)\}}\bigg] + \Enl\bigg[\sum_{|S|=k} \sum_{\substack{|S'|=k,\\ S'\cap S=\emptyset}} \ind_{\{S\in \mathcal C(k,k+l)\}}\ind_{\{S'\in \mathcal C(k,k+l)\}}\bigg]\]
and if either $j\neq k$ or $l\neq l'$,
\[\Enl[X(j,j+l)X(k,k+l')] = \Enl\bigg[\sum_{|S|=j} \sum_{\substack{|S'|=k,\\ S'\cap S=\emptyset}} \ind_{\{S\in \mathcal C(j,j+l)\}}\ind_{\{S'\in \mathcal C(k,k+l')\}}\bigg].\]
By the argument that follows \eqref{eq:EX}, this last quantity equals (for any $j,k,l,l'$)
\[\binom{n}{j} \binom{n-j}{k} C(j,j+l)C(k,k+l') \pnl^{j+l+k+l'} (1-\pnl)^{\binom{j}{2} + \binom{k}{2} - j-l-k-l' + j(n-j) + k(n-j-k)}.\]
By comparing with the formula \eqref{eq:EX} for $\Enl[X(k,k+l)]$, we see that
\begin{equation*}
\Enl[X(k,k+l)^2] = \Enl[X(k,k+l)] + \Enl[X(k,k+l)]^2 \frac{\binom{n-k}{k}}{\binom{n}{k}} (1-\pnl)^{-k^2}
\end{equation*}
and if either $j\neq k$ or $l\neq l'$, then
\begin{equation*}
\Enl[X(j,j+l)X(k,k+l')] = \Enl[X(j,j+l)]\Enl[X(k,k+l')] \frac{\binom{n-j}{k}}{\binom{n}{k}} (1-\pnl)^{-jk}.
\end{equation*}
Therefore it suffices to show that for any $j,k\le n^{3/4}\wedge{\frac{n}{3}}$,
\begin{equation}\label{eq:quotientreq}
\frac{\binom{n-j}{k}}{\binom{n}{k}} (1-\pnl)^{-jk} = \exp\left(\frac{\lambda jk}{n^{4/3}} - \frac{jk}{2n^2}(j+k)\right) (1+O(\tfrac{j+j'}{n}+\tfrac{j^4+k^4}{n^3})).
\end{equation}

Without loss of generality we may assume that $n\ge 81$ and then since $j,k\le n^{3/4}$ we also have $j,k\le n/3$. Stirling's formula tells us that
\begin{align*}
\frac{\binom{n-j}{k}}{\binom{n}{k}} = \frac{(n-j)!(n-k)!}{(n-j-k)!n!} &= \frac{(n-j)^{n-j+1/2}(n-k)^{n-k+1/2}}{n^{n+1/2}(n-j-k)^{n-j-k+1/2}}(1+O(1/n))\\
&= \left(1-\frac{j}{n}\right)^{n-j}\left(1+\frac{j}{n-j-k}\right)^{n-j-k}\left(1-\frac{k}{n}\right)^j (1+O(\tfrac{j+k}{n})).
\end{align*}
Using the expansion \eqref{log},
\[\Big(1-\frac{j}{n}\Big)^{n-j} = \exp\Big(-j + \frac{j^2}{2n} + \frac{j^3}{6n^2} + O(j^4/n^3)\Big),\]
\begin{align*}
\Big(1+\frac{j}{n-j-k}\Big)^{n-j-k} &= \exp\Big(j - \frac{j^2}{2(n-j-k)} + \frac{j^3}{3(n-j-k)^2} + O(\tfrac{j^4}{n^3})\Big)\\
&= \exp\Big(j - \frac{j^2}{2n} - \frac{j^2}{2n^2}(j+j') + \frac{j^3}{3n^2} + O(\tfrac{j^4}{n^3})\Big),
\end{align*}
and
\[\Big(1-\frac{k}{n}\Big)^j = \exp\Big(-\frac{jk}{n} - \frac{jk^2}{2n^2} + O(\tfrac{jk^3}{n^3})\Big).\]
Combining these, we get
\[\frac{\binom{n-j}{k}}{\binom{n}{k}} = \exp\Big(-\frac{jk}{n} - \frac{jk}{2n^2}(j+k)\Big)(1+O(\tfrac{j+k}{n}+\tfrac{j^4+k^4}{n^3})).\]
On the other hand,
\[(1-\pnl)^{-jk} = \exp\left(\frac{jk}{n} + \frac{\lambda jk}{n^{4/3}} + O(\tfrac{jk}{n^2})\right).\]
These two approximations establish \eqref{eq:quotientreq} and complete the proof.
\end{proof}

For $k\in\mathbb{N}$, recall that we defined $Y(k)$ to be the number of components of size exactly $k$, and for $k\le n^{3/4}$, we let $Z(k)$ be the number of components of size between $k$ and $n^{3/4}$. We can use Lemma \ref{le:EX2toEX1} to bound the second moments of $Y(k)$ and $Z(k)$.

\begin{lem}\label{le:2mom}
If $k\le n^{3/4}$, then
\[\Enl[Y(k)^2] = \Enl[Y(k)] + \Enl[Y(k)]^2 \exp\left(\frac{\lambda k^2}{n^{4/3}} - \frac{k^3}{n^2}\right) (1+O(\tfrac{k}{n}+\tfrac{k^4}{n^3}))\]
and for $j\neq k$ with $j\le n^{3/4}$,
\[\Enl[Y(j)Y(k)] = \Enl[Y(j)]\Enl[Y(k)] \exp\left(\frac{\lambda jk}{n^{4/3}} - \frac{jk}{2n^2}(j+k)\right) (1+O(\tfrac{j+k}{n}+\tfrac{j^4+k^4}{n^3})).\]
Further, if $\lambda n^{2/3} \le k\le n^{3/4}$, then
\[\Enl[Z(k)^2] \le \Enl[Z(k)] + \Enl[Z(k)]^2 \exp\left(\frac{\lambda k^2}{n^{4/3}} - \frac{k^3}{n^2}\right) (1+O(n^{-1/11})).\]
\end{lem}

\begin{proof}
We begin with $\Enl[Y(k)^2]$. Clearly
\begin{align*}
\Enl[Y(k)^2] &= \sum_{l,l'\ge -1} \Enl[X(k,k+l)X(k,k+l')]\\
&= \sum_{l\ge -1} \Enl[X(k,k+l)^2] + \sum_{l\ge -1}\sum_{\substack{l'\ge -1,\\ l'\neq l}} \Enl[X(k,k+l)X(k,k+l')].
\end{align*}
By Lemma \ref{le:EX2toEX1}, this equals
\[\sum_{l\ge -1} \Enl[X(k,k+l)] + \sum_{l,l'\ge -1} \Enl[X(k,k+l)]\Enl[X(k,k+l')]\exp\left(\frac{\lambda k^2}{n^{4/3}} - \frac{k^3}{n^2}\right) (1+O(\tfrac{k}{n}+\tfrac{k^4}{n^3})).\]
We now recognise $\Enl[Y(k)]=\sum_{l\ge -1}\Enl[X(k,k+l)]$ and the first part of the lemma follows.

Similarly, if $j\neq k$,
\begin{align*}
&\Enl[Y(j)Y(k)]\\
&\hspace{5mm} = \sum_{l,l'\ge -1}\Enl[X(j,j+l)X(k,k+l)]\\
&\hspace{5mm} = \sum_{l,l'} \Enl[X(j,j+l)]\Enl[X(k,k+l)] e^{\lambda jk/n^{4/3} - jk(j+k)/(2n^2)} (1+O(\tfrac{j+k}{n}+\tfrac{j^4+k^4}{n^3}))\\
&\hspace{5mm} = \Enl[Y(j)]\Enl[Y(k)] e^{\lambda jk/n^{4/3} - jk(j+k)/(2n^2)} (1+O(\tfrac{j+k}{n}+\tfrac{j^4+k^4}{n^3})).
\end{align*}
This establishes the second part of the lemma.

For $\Enl[Z(k)^2]$, we let $N=\lfloor n^{3/4}\rfloor$ and write
\begin{align*}
\Enl[Z(k)^2] &= \sum_{i=k}^N \sum_{j=k}^N \Enl[Y(i)Y(j)]\\
&= \sum_{i=k}^N \E[Y(i)^2] + \sum_{i=k}^N \sum_{\substack{j=k,\ldots,N,\\ j\neq i}} \Enl[Y(i)Y(j)].
\end{align*}
By the first two parts of the lemma, this equals
\[\sum_{i=k}^N \E[Y(i)] + \sum_{i,j=k}^N \E[Y(i)]\E[Y(j)] \exp\left(\frac{\lambda ij}{n^{4/3}} - \frac{ij}{2n^2}(i+j)\right) (1+O(\tfrac{i+i}{n}+\tfrac{i^4+j^4}{n^3})).\]
Noting that if $\lambda n^{2/3} \le k$ then the exponent above is decreasing in $i$ and $j$ for $i,j\ge k$, we get
\[\Enl[Z(k)^2] \le \sum_{i=k}^N \E[Y(i)] + \sum_{i,j=k}^N \E[Y(i)]\E[Y(j)] \exp\left(\frac{\lambda k^2}{n^{4/3}} - \frac{k^3}{n^2}\right) (1+O(\tfrac{N}{n}+\tfrac{N^4}{n^3})).\]
Again recognising $\Enl[Z(k)] = \sum_{i=k}^N \E[Y(i)]$, we get the desired result.
\end{proof}

\section{Large $k$: a simple exploration process argument}\label{sec:explore}

The next step is to prove Lemma \ref{le:explore}. To bound $\Pnl(L_1>k)$ when $k$ is large (say $k\ge n^{3/4}$) it is easiest to use a different approach from our usual combinatorial arguments. We do not aim to give best possible bounds, and instead extract an argument from \cite{nachmias_peres:CRG_mgs}. An improved bound (in particular not including the factor of $n^{1/4}$) could be obtained by more closely following the proof in \cite{nachmias_peres:CRG_mgs}.

\begin{proof}[Proof of Lemma \ref{le:explore}]
We use the following exploration process: start with one vertex, $v$, in a queue. At each step, choose a vertex from the queue, remove it from the queue, and add to the queue all of its neighbours that have not previously been in the queue. Stop when the queue is empty; at this stage the set of all vertices that have been in the queue at any step is exactly the set of vertices of $\Cc(v)$, the connected component containing $v$.

Let $V_j$ be the vertex chosen from the queue at step $j$. Let $\eta_j$ be the number of neighbours of $V_j$ that have not been in the list up to step $j$, and $B_j$ the number of neighbours of $V_j$ that are not in the set $\{V_1,\ldots,V_{j-1}\}$. Note that for each $j\ge 1$,
\begin{itemize}
\item $\eta_j$ is the number of vertices added to the queue at step $j$;
\item $\eta_j-1$ is the change in the length of the queue at step $j$;
\item $\eta_j \le B_j$;
\item the random variable $B_j$ is binomially distributed with parameters $n-j$ and $\pnl$, and is independent of $B_1,\ldots,B_{j-1}$.
\end{itemize}
Let $Z_k =\sum_{j=1}^k (\eta_j-1)$ and $T=\min\{k\ge1 : Z_k=0\}$. Then $T$ is the first step at which the queue is empty, and since we removed one vertex at each step, we must have $|\Cc(v)|=T$. Thus for any $\mu>0$,
\begin{equation}\label{eq:Zmarkov}
\Pnl(|\Cc(v)| > k) \le \Pnl(Z_{k} \ge 1) = \Pnl(e^{\mu Z_{k}} \ge e^\mu) \le \Enl[e^{\mu Z_{k}}].
\end{equation}
The random variables $\eta_1, \eta_2,\ldots$ are not independent, so we work instead with the independent binomial random variables $B_1, B_2,\ldots$ mentioned above. Then for any $\mu > 0$ we have
\[\Enl[e^{\mu(B_j-1)}] = e^{-\mu}(1+(e^\mu-1)\pnl)^{n-j}.\]
For $\mu\le 1$, $e^\mu - 1 \le \mu + \mu^2$, and for any $x$, $1+x\le e^x$. Therefore for $\mu\le 1$ the above is at most
\[e^{-\mu} e^{(n-j)\mu\pnl + (n-j)\mu^2\pnl} \le \exp\big(-\mu(1-n\pnl) - j\mu\pnl + \mu^2 n\pnl\big).\]
Recalling that $\pnl = \frac{1}{n}(1+\frac{\lambda}{n^{1/3}})$, we obtain
\[\Enl[e^{\mu(B_j-1)}] \le \exp\Big(\frac{\mu\lambda}{n^{1/3}} - j\mu\pnl + \mu^2 n \pnl \Big).\]
Since $Z_k = \sum_{j=1}^k (\eta_j-1) \le \sum_{j=1}^k (B_j-1)$, substituting this estimate into \eqref{eq:Zmarkov} we have that for any $\mu\in(0,1]$,
\[\Pnl(|\Cc(v)|> k) \le \prod_{j=1}^{k} \exp\Big(\frac{\mu\lambda}{n^{1/3}} - j\mu\pnl + \mu^2 n \pnl \Big) \le \exp\Big(\frac{k\mu\lambda}{n^{1/3}} - \frac{k^2}{2} \mu \pnl + k \mu^2 n\pnl\Big).\]
Choosing $\mu=k/(4n)$, we get
\[\Pnl(|\Cc(v)|> k) \le \exp\Big(\frac{k^2\lambda}{4n^{4/3}} - \frac{k^3 \pnl}{8n} + \frac{k^3 \pnl}{16n}\Big) = \exp\Big(-\frac{k^3}{16n^2}\Big(1+\frac{\lambda}{n^{1/3}} - \frac{4 n^{2/3}\lambda}{k}\Big).\]
If $k\ge n^{3/4}$ and $\lambda\le n^{1/12}/5$ then $4n^{2/3}\lambda/k\le 4/5$, so finally
\[\Pnl(|\Cc(v)|> k) \le \exp\Big(-\frac{k^3}{16n^2}\Big(\frac{1}{5}+\frac{\lambda}{n^{1/3}}\Big)\Big).\]
Clearly the probability on the left is increasing in $\lambda$, and the right-hand side is decreasing in $\lambda$, so without loss of generality we may take $\lambda=0$ (we would get a slightly better bound if we took $\lambda$ as large as possible, namely $\lambda=n^{1/12}/5$, but we choose $\lambda=0$ for simplicity). Finally,
\begin{align*}
\Pnl(L_1> k) \le \Pnl(\#\{u : |\Cc(u)|> k\}> k) &\le \frac{1}{k}\Enl[\#\{u: |\Cc(u)|> k\}]\\
&= \frac{n}{k} \P(|\Cc(v)|> k) \le n^{1/4}\exp\Big(-\frac{k^3}{80 n^2}\Big).\qedhere
\end{align*}
\end{proof}

\section{Proof of Theorem \ref{thm:YZ}: moment bounds to probabilities}\label{sec:momtoprob}

In this section we put together the moment estimates that we proved in previous sections to complete the proof of Theorem \ref{thm:YZ}. We will at several points use the inequality
\begin{equation}\label{eq:2m}
\P( V \ge 1 ) \ge \frac{\E[V]^2}{\E[V^2]}
\end{equation}
which holds for any non-negative integer-valued random variable $V$ and is proved by applying the Cauchy-Schwarz inequality to $\E[V\ind_{\{V\ge 1\}}]$. Define
\begin{equation}\label{eq:Pmerr}
\Cl[error]{Pmerr} = \Cr{Pmerr}(k,n,\lambda) = \frac{n}{k^{3/2}} e^{-G_\lambda(k/n^{2/3}) + \lambda k^2/n^{4/3} - k^3/n^2}.
\end{equation}
We now give two lemmas that relate the probabilities of events that appeared in Theorem \ref{thm:YZ} with the expectations that we calculated in Propositions \ref{prop:Y} and \ref{prop:Z}.

\begin{lem}\label{le:PtoY}
Suppose that $-n^{1/12}\le \lambda\le n^{1/12}/5$ and $(3\lambda\wedge 1)n^{2/3}\le k \le n^{3/4}$. Then
\[\Pnl(L_1 = k) = \Enl[Y(k)](1+O(\Cr{Pmerr})) + O(\Cr{A1err}).\]
\end{lem}

\begin{proof}
Clearly $\Pnl(L_1=k)\le \Pnl(Y(k)\ge 1)\le \Enl[Y(k)]$, so we may concentrate on the lower bound. Let $N=\lfloor n^{3/4}\rfloor$. We have
\begin{align*}
\P(L_1=k) &= \P(Y(k)\ge 1) - \P(Y(k)\ge 1,\, \exists j\ge k \hbox{ such that } Y(j)\ge 1)\\
&\ge \P(Y(k)\ge 1) - \sum_{j=k+1}^N \P(Y(k)\ge 1,\, Y(j)\ge 1) - \P(L_1 > N)\\
&\ge \frac{\E[Y(k)]^2}{\E[Y(k)^2]} - \sum_{j=k+1}^N \E[Y(k)Y(j)] - \P(L_1>N) 
\end{align*}
where for the last line we have applied \eqref{eq:2m} and Markov's inequality.

By Lemma \ref{le:2mom} we have
\[\Enl[Y(k)^2] = \Enl[Y(k)]\big(1+O\big(\Enl[Y(k)]e^{\lambda k^2/n^{4/3}-k^3/n^2}\big)\big),\]
so
\[\frac{\E[Y(k)]^2}{\E[Y(k)^2]} = \Enl[Y(k)]\big(1+O\big(\Enl[Y(k)]e^{\lambda k^2/n^{4/3}-k^3/n^2}\big)\big).\]
Lemma \ref{le:2mom} also tells us that for $j>k$,
\[\Enl[Y(k)Y(j)] = \Enl[Y(k)]\Enl[Y(j)] e^{\lambda jk/n^{4/3} - jk(j+k)/(2n^2)}(1+ O(1)),\]
so for $\lambda \le k/n^{2/3}$,
\begin{align*}
\sum_{j=k+1}^N \Enl[Y(k)Y(j)] &\le \Enl[Y(k)]\sum_{j=k+1}^N \Enl[Y(j)] e^{\lambda k^2/n^{4/3}-k^3/n^2}(1+ O(1))\\
&= \Enl[Y(k)]\Enl[Z(k+1)] e^{\lambda k^2/n^{4/3}-k^3/n^2} (1+O(1)).
\end{align*}
Thus (bounding both $\Enl[Y(k)]$ and $\Enl[Z(k+1)]$ above by $\Enl[Z(k)]$)
\[\P(L_1=k) \ge \Enl[Y(k)]\big(1+O\big(\Enl[Z(k)] e^{\lambda k^2/n^{4/3}-k^3/n^2}\big)\big) - \P(L_1>N).\]
Lemma \ref{le:explore} gives
\[\Pnl(L_1>N) \le n^{1/4}\exp(-n^{1/4}/80)=\Cr{A1err},\]
and Proposition \ref{prop:Z} gives
\[\Enl[Z(k)] = O\Big(\frac{n}{k^{3/2}} e^{-G_\lambda(k^{2/3}/n)}\Big).\]
The result follows.
\end{proof}

\begin{lem}\label{le:PtoZ}
Suppose that $-n^{1/12}\le \lambda\le n^{1/12}/5$ and $(3\lambda\wedge 1)n^{2/3}+1\le k \le n^{3/4}$. Then
\[\Pnl(L_1 \ge k) = \Enl[Z(k)](1+O(\Cr{Pmerr})) + O(\Cr{A1err}).\]
\end{lem}

\begin{proof}
First note that
\[\Pnl(L_1\ge k) \le \Pnl(Z(k)\ge 1) + \Pnl(L_1> N) \le \Enl[Z(k)] + \Pnl(L_1> N).\]
Lemma \ref{le:explore} tells us that $\Pnl(L_1> N)\le \Cr{A1err}$, so the upper bound is done. For the lower bound, by \eqref{eq:2m},
\[\Pnl(L_1\ge k) \ge \Pnl(Z(k)\ge 1) \ge \frac{\Enl[Z(k)]^2}{\Enl[Z(k)^2]}.\]
Lemma \ref{le:2mom} tells us that
\[\Enl[Z(k)^2] \le \Enl[Z(k)] + \Enl[Z(k)]^2 \exp\left(\frac{\lambda k^2}{n^{4/3}} - \frac{k^3}{n^2}\right) (1+O(n^{-1/11})),\]
and Proposition \ref{prop:Z} gives
\[\Enl[Z(k)] = O\Big(\frac{n}{k^{3/2}} e^{-G_\lambda(k^{2/3}/n)}\Big).\]
Thus
\[\frac{\Enl[Z(k)]^2}{\Enl[Z(k)^2]} \ge \Enl[Z(k)] \Big(1+O\Big(\frac{n}{k^{3/2}} e^{-G_\lambda(k^{2/3}/n)+\lambda k^2/n^{4/3} - k^3/n^2}\Big)\Big) = \Enl[Z(k)](1+O(\Cr{Pmerr}))\]
as required.
\end{proof}

Given these lemmas, our first main result follows almost trivially.

\begin{proof}[Proof of Theorem \ref{thm:YZ}]
Lemma \ref{le:PtoY} tells us that
\[\Pnl(L_1 = k) = \Enl[Y(k)](1+O(\Cr{Pmerr})) + O(\Cr{A1err}).\]
But by Proposition \ref{prop:Y},
\[\Enl[Y(k)] = \frac{k^{1/2}}{(8\pi)^{1/2}n} e^{-G_\lambda(k/n^{2/3})} (1+O(\Cr{M1err})).\]
Since $G_\lambda(x)\ge 0$ for all $x\ge 0$, and by the conditions of the theorem we have $\lambda\le k/n^{2/3}$, we see that $\Cr{Pmerr}\le n/k^{3/2}\le \Cr{M1err}$. This gives part (a) of the result.

For part (b), Lemma \ref{le:PtoZ} gives us
\[\Pnl(L_1 \ge k) = \Enl[Z(k)](1+O(\Cr{Pmerr})) + O(\Cr{A1err}).\]
Then by Proposition \ref{prop:Z},
\[\Enl[Z(k)] = (1+O(\Cr{M1err})) \frac{1}{(8\pi)^{1/2}} \frac{(k/n^{2/3})^{1/2}}{G'_\lambda(k/n^{2/3})} e^{-G_\lambda(k/n^{2/3})} + O(\Cr{A1err}).\]
Again we observe that $\Cr{Pmerr}\le \Cr{M1err}$ which completes the proof.
\end{proof}

\section{The size of the component containing a particular vertex}\label{sec:add}

We now turn to proving Theorem \ref{thm:vLYZ}, which concerns the component containing a particular vertex $v$. The proof will rely only on the first moment estimates that we have already developed. Along the way we will also prove Propositions \ref{prop:vl} and \ref{prop:vsmallk}. We recall that $\Cc(v)$ is the connected component containing $v$, $|\Cc(v)|$ is the number of vertices in $\Cc(v)$ and $E(\Cc(v))$ is the number of edges in $\Cc(v)$.

Everything will be based on the following elementary lemma.

\begin{lem}\label{le:vtoX}
For any $n,k\ge 1$ and $l\ge -1$,
\[\Pnl\big(|\Cc(v)|=k,\, E(\Cc(v))=k+l\big) = \frac{k}{n}\Enl[X(k,k+l)].\]
\end{lem}

\begin{proof}
We merely note that, by the same argument as for \eqref{eq:EX},
\begin{align*}
\Pnl\big(|\Cc(v)|=k,\, E(\Cc(v))=k+l\big) &= \binom{n-1}{k-1} C(k,k+l) \pnl^{k+l} (1-\pnl)^{\binom{k}{2}-k-l+(n-k)k}\\
&= \frac{k}{n}\Enl[X(k,k+l)].\qedhere
\end{align*}
\end{proof}

\begin{proof}[Proof of Proposition \ref{prop:vl}]
Simply combine Lemma \ref{le:vtoX} with Corollary \ref{cor:EX}.
\end{proof}

\begin{proof}[Proof of Proposition \ref{prop:vsmallk}]
Let $L=\lfloor 4n^{1/4}\rfloor$. By Lemma \ref{le:vtoX},
\[\Pnl(|\Cc(v)|=k) = \sum_{l=-1}^L \frac{k}{n}\Enl[X(k,k+l)] + \sum_{l=L+1}^{\binom{k}{2}-k} \frac{k}{n}\Enl[X(k,k+l)].\]
By the first part of Proposition \ref{prop:main},
\[\sum_{l=-1}^L \frac{k}{n}\Enl[X(k,k+l)] = \sum_{l=-1}^L \frac{k}{n}\frac{C(k,k+l)}{(2\pi)^{1/2} n^l k^{k+1/2}} e^{-F_\lambda(k/n^{2/3})}(1+O(\Cr{EXerr})).\]
Since $\Cr{EXerr}$ is at most a constant, this means that there exist $0<c_1<c_2<\infty$ such that
\[\frac{c_1 e^{-F_\lambda(k/n^{2/3})}}{n k^{k-1/2}} \sum_{l=-1}^L \frac{C(k,k+l)}{n^l} \le \sum_{l=-1}^L \frac{k}{n}\Enl[X(k,k+l)] \le \frac{c_2 e^{-F_\lambda(k/n^{2/3})}}{n k^{k-1/2}} \sum_{l=-1}^L \frac{C(k,k+l)}{n^l}.\]
Corollary \ref{cor:smallk} then gives
\[c_1 k^{-3/2}e^{-F_\lambda(k/n^{2/3})} \le \sum_{l=-1}^L \frac{k}{n}\Enl[X(k,k+l)] \le c_2 k^{-3/2}e^{-F_\lambda(k/n^{2/3})}.\]
Thus it remains to show that $\sum_{l=L+1}^\infty \frac{k}{n}\Enl[X(k,k+l)]$ is of smaller order. But this follows from the second part of Proposition \ref{prop:main}.
\end{proof}

\begin{proof}[Proof of Theorem \ref{thm:vLYZ}]
By Lemma \ref{le:vtoX},
\[\Pnl(|\Cc(v)|=k) = \frac{k}{n}\sum_{l=-1}^{\binom{k}{2}-k} \Enl[X(k,k+l)] = \frac{k}{n} \Enl[Y(k)].\]
Part (a) now follows from Proposition \ref{prop:Y}.

For part (b) of the theorem, let $N=\lfloor n^{3/4}\rfloor$. By Lemma \ref{le:vtoX},
\[\Pnl(|\Cc(v)|\ge an^{2/3},\, L_1 \le N) = \sum_{j=k}^N \frac{j}{n}\sum_{l\ge -1} \Enl[X(j,j+l)] = \sum_{j=k}^N \frac{j}{n}\Enl[Y(j)].\]
By Proposition \ref{prop:Y}, this equals
\begin{multline*}
\sum_{j=k}^N \frac{j^{3/2}}{(8\pi)^{1/2}n^2} e^{-G_\lambda(j/n^{2/3})}(1+O(\Cr{M1err}(j,n,\lambda)))\\
= \frac{1}{(8\pi)^{1/2}n}\sum_{j=k}^N \Big(\frac{j}{n^{2/3}}\Big)^{3/2} e^{-G_\lambda(j/n^{2/3})}(1+O(\Cr{M1err}(j,n,\lambda))).
\end{multline*}
Applying Lemma \ref{le:Zsumtoint}, this is
\[\frac{1}{(8\pi)^{1/2}n^{1/3}} \frac{(k/n^{2/3})^{3/2}}{G'_\lambda(k/n^{2/3})} e^{-G_\lambda(k/n^{2/3})} (1+O(\Cr{M1err}(k,n,\lambda))).\]
By Lemma \ref{le:explore}, $\Pnl(L_1>N) = O(\Cr{A1err}(n))$, so we have shown that
\[\Pnl(|\Cc(v)|\ge an^{2/3}) = \frac{1}{(8\pi)^{1/2}n^{1/3}}\frac{a^{3/2}}{G'_\lambda(a)} e^{-G_\lambda(a)} (1+O(\Cr{M1err}(k,n,\lambda))) + O(\Cr{A1err}(n))\]
which is (b).
\end{proof}

\subsection*{Acknowledgements}
Thanks to Bat\i{} \c{S}eng\"ul for several helpful conversations, and to two anonymous referees for their valuable comments and suggestions. I am grateful also for support from EPSRC fellowship EP/K007440/1, and a Royal Society University Research Fellowship.

\bibliographystyle{plain}
\def\cprime{$'$}

\end{document}